\documentclass{article}
\usepackage[utf8]{inputenc}
\usepackage[margin=1.25in]{geometry}
\usepackage{amsmath}
\usepackage{amssymb}
\usepackage{amsthm}
\usepackage{blindtext}
\usepackage{titlesec}
\usepackage{mathtools}
\usepackage{esint}
\usepackage{graphicx}
\graphicspath{ {images/} }
\usepackage{titlesec}
\usepackage{biblatex}
\usepackage{ragged2e}

\title{Ancient Solutions to the Biharmonic Heat Equation}
\author{Alexander D. McWeeney}
\date{}

\newtheorem{theorem}{Theorem}[section]
\newtheorem*{theorem*}{Theorem}
\newtheorem{corollary}{Corollary}[section]
\newtheorem{lemma}[theorem]{Lemma}
\newtheorem{proposition}[theorem]{Proposition}
\theoremstyle{definition}

\theoremstyle{remark}

\DeclareMathOperator{\Ric}{\operatorname{Ric}}

\begin{document}
\DeclareNameAlias{author}{family-given}
\numberwithin{equation}{section}

\maketitle
\begin{abstract}
    We show that the space of polynomially bounded ancient solutions to the biharmonic heat equation on a complete manifold with polynomial volume growth is bounded by the dimensions of spaces of polynomially bounded biharmonic functions. This generalizes the work of Colding and Minicozzi in \cite{colding2021optimal} for ancient caloric functions.
\end{abstract}

\section{Introduction}
The relationship between the geometry of manifolds and the analytic properties of functions on manifolds is a defining theme of geometric analysis. Our direction starts with the Liouville theorems for harmonic functions on \(\mathbb{R}^n\) and Yau's generalization.

Yau proved in \cite{yau1975harmonic} that a bounded harmonic function on a complete manifold with nonnegative Ricci curvature is a constant. In 1974, he conjectured that a more general result should hold: on a complete manifold \(M\) with nonnegative Ricci curvature, the space \(\mathcal{H}_d(M)\) of harmonic functions with polynomially bounded growth should have finite dimension. Colding and Minicozzi proved his conjecture in \cite{colding1997harmonic}. 

A natural generalization is to try to show this result for solutions of the heat equation. However, the heat equation is very flexible compared to the Laplace equation, and since there are bounded solutions to the heat equation, no Liouville theorem is possible in general.

Despite this, if we restrict attention to specifically \textit{ancient} solutions of the heat equation, that is, solutions which are defined for all time going back to \(-\infty\), then Liouville theorems actually do become possible. Indeed, in \cite{colding2021optimal}, Colding and Minicozzi generalize \cite{colding1997harmonic} to show that the space \(\mathcal{P}_d(M)\) of ancient solutions of the heat equation with polynomially bounded growth also has finite dimension. In (\cite{colding2020complexity}, \cite{colding2019search}, \cite{colding2019liouville}), Colding and Minicozzi show how the spaces \(\mathcal{P}_d(M)\) are relevant to geometric flows.

Continuing to more types of equations, Wang and Zhu recently generalized the result of \cite{colding1997harmonic} to biharmonic functions \cite{wang2025qualitativebehaviorbiharmonicfunctions}, i.e., functions \(u: M \rightarrow \mathbb{R}\) solving
\begin{equation*}
    \Delta\Delta u = 0.
\end{equation*}

This equation is also more flexible than the Laplace equation (indeed, any harmonic function is biharmonic), and we cannot prove as general a Liouville theorem as for harmonic functions. To find a Liouville theorem, rather than restricting attention to a subclass of biharmonic functions as in \cite{colding2021optimal}, Wang and Zhu instead restrict attention to a subclass of manifolds with polynomial volume growth and Ricci curvature bounded below at infinity.

Our goal in this paper is to generalize Wang and Zhu's result to ancient solutions of the biharmonic heat equation, following the strategy of Colding and Minicozzi in \cite{colding2021optimal}. Our main result is
\begin{theorem*}
    Let \(M\) be a complete Riemannian manifold with polynomial volume growth and Ricci curvature bounded below quadratically. Let \(u: M \times (-\infty, 0] \rightarrow \mathbb{R}\) be an ancient solution of
    \begin{equation*}
        \partial_t u(x, t) + \Delta\Delta u(x, t) = 0
    \end{equation*}
    such that \(|u(x, t)|\) and \(|\nabla u(x, t)|\) have polynomially bounded growth in the heat balls \(B_R(x) \times [-R^4, 0]\). The space of all such solutions \(u(x, t)\) is finite dimensional.
\end{theorem*}

\subsection{Definitions and Notation}
We now give more precise definitions and statements. Given a manifold \(M\) and an interval \(I \subset \mathbb{R}\), a function \(u: M \times I \rightarrow \mathbb{R}\) satisfies the biharmonic heat equation if
\begin{equation}\label{first instance of biharmonic heat equation}
    \partial_t u(x, t) + \Delta\Delta u(x, t) = 0.
\end{equation}
We will call such a function ``bicaloric'' for brevity. A bicaloric function \(u\) is ancient if it can be defined on an interval extending infinitely backwards in time, i.e. for \(t \in (-\infty, 0]\). We say that \(u \in \mathcal{P}_{d, d'}(M)\) for \(d, d' > 0\) if \(\partial_t u + \Delta\Delta u = 0\), \(u\) is ancient, and for some constants \(C, C' > 0\),
\begin{equation}\label{definition of polynomially bounded growth in introduction}
    \sup_{B_R(p) \times [-R^4, 0]}|u(x, t)| \le C(1 + R)^d, \quad \sup_{B_R(p) \times [-R^4, 0]}|\nabla u(x, t)| \le C'(1 + R)^{d'}
\end{equation}
for any \(p \in M\) and \(R > 0\). We similarly say that \(u \in \mathcal{H}_{d, d'}(M)\) if \(\Delta\Delta u = 0\) and the same bounds in (\ref{definition of polynomially bounded growth in introduction}) hold, where we take the supremum over only the ball \(B_R(p)\).

A manifold \(M\) is said to have polynomial volume growth if there are constants \(C, d_V > 0\) and some \(p \in M\) such that \(\operatorname{Vol}(B_R(p)) \le C(1 + R)^{d_V}\) for all \(R> 0\). Furthermore, we say that the Ricci curvature tensor is bounded below quadratically with constant \(K\) if for some \(p \in M\) and all \(R > 0\),
\begin{equation}
    \sup_{v \in TB_{R}(p)}\frac{\Ric(v, v)}{|v|^2} \ge -\frac{K}{R^2}.
\end{equation}

With these definitions, our main results are more precisely stated as
\begin{theorem}\label{dimension bound theorem}
    Let \(M\) be a complete Riemannian manifold with polynomial volume growth and Ricci curvature bounded below quadratically. Let \(k\), \(\ell\) be nonnegative integers. Then
    \begin{equation}\label{equation in the first statement of the main theorem}
        \dim \mathcal{P}_{4k, 4\ell}(M) \le \begin{cases}
             \displaystyle \sum_{i = 0}^{k} \dim \mathcal{H}_{4(k - i), 4(\ell - i)}(M) & \quad k \le \ell + 1, \\ \displaystyle
            1 + \sum_{i = 0}^{\ell} \dim \mathcal{H}_{4(k - i), 4(\ell - i)}(M) & \quad k > \ell + 1
        \end{cases}
    \end{equation}
    Moreover, these inequalities are sharp in \(\mathbb{R}^n\).
\end{theorem}

Combining this with Wang and Zhu's result \cite{wang2025qualitativebehaviorbiharmonicfunctions}, we have the following corollary:

\begin{corollary}
    Let \(M\) be a Riemannian manifold with polynomial volume growth and Ricci curvature bounded below quadratically. Then for \(k, \ell \ge 0\) the spaces \(\mathcal{P}_{4k, 4\ell}(M)\) are finite dimensional.
\end{corollary}

\subsection{Harmonic and biharmonic functions}

Biharmonic functions arise in several variational problems. Just as minimizing \(\int|\nabla u|^2\) leads one to the Laplace and heat equations, minimizing \(\int |\Delta u|^2\) leads to the biharmonic and biharmonic heat equations.

In general, fourth order elliptic operators arise naturally when taking variations involving second order objects, one major example being variations of metrics in conformal geometry (see \cite{chang1995extremal}, \cite{Lin1998classification}). They also arise in the study of the Willmore energy. For an immersed surface \(\phi: M^2 \rightarrow \mathbb{R}^3\), the Willmore energy is defined as
\begin{equation}\label{Willmore functional}
    \mathcal{W}(\phi) = \int_{M}H^2\,dA
\end{equation}
where \(dA\) is the induced volume element and \(H\) is the mean curvature \cite{willmore2000surfaces}. In studying critical points of this functional one arrives at the Euler-Lagrange equation
\begin{equation}\label{euler-lagrange equation}
    \Delta H + 2H(H^2 - K) = 0,
\end{equation}
a fourth order elliptic operator. The biharmonic heat equation similarly arises when studying the gradient flow of the Willmore energy (\cite{kuwert2002gradient}, \cite{lamm2005biharmonic}). Ancient solutions to heat equations often appear when doing blowup analysis of general solutions to a variational problem. See \cite{kuwert2004removable} for blowup analysis of singularities of Willmore flows. We also again reference (\cite{colding2020complexity}, \cite{colding2019search}, \cite{colding2019liouville}) for more on how ancient solutions to heat equations with polynomially bounded growth are relevant to geometric flows.

Although both arise from variational problems, biharmonic functions in general differ significantly from harmonic functions, because no maximum principle holds for biharmonic functions. This limits the kinds of estimates we can find for biharmonic functions. In particular, the usual pointwise derivative estimates one can find for harmonic functions on a ball cannot be found for a biharmonic function.

On the bright side, energy methods for harmonic and caloric functions seem to have analogs for biharmonic and bicaloric functions, which we will see as we prove Theorem \ref{dimension bound theorem}. We are still limited to some extent, however, because when performing integrations by parts we are forced to use the Bochner formula
\begin{equation*}
    \frac{1}{2}\Delta|\nabla u|^2 = |\nabla^2u| + \langle \nabla \Delta u, \nabla u\rangle + \Ric(\nabla u, \nabla u)
\end{equation*}
to control the factor \(\langle \nabla \Delta u, \nabla u\rangle\). It is the appearance of the Ricci term here that makes the decay on Ricci curvature crucial for our result.

Our methodology is inspired by Colding and Minicozzi's in \cite{colding2021optimal}. We will show a reverse Poincaré inequality for bicaloric functions on ``heat balls'' \(B_R(p) \times [-R^4, 0]\). Because we are considering ancient bicaloric functions, we will be able to apply the inequality as \(R \rightarrow \infty\) to get strong, global control of their behavior. In particular we will see that high order time derivatives \(\partial_t^ku\) must vanish identically, allowing us to write for some finite \(d\):
\[u(x, t) = p_d(x)t^d + \cdots + p_1(x)t + p_0(x)\]
with \(\Delta \Delta p_d = 0\) and \(\Delta \Delta p_j = -(j + 1)p_{j + 1}\). This will allow us to directly compare the spaces \(\mathcal{H}_{4k, 4\ell}(M)\) with \(\mathcal{P}_{4k, 4\ell}(M)\).

To show the dimension estimates are sharp in \(\mathbb{R}^n\), we will consider biharmonic and bicaloric polynomials (analogs of the harmonic polynomials), enabling us to explicitly compute the dimensions of the spaces \(\mathcal{H}_{4k, 4\ell}(\mathbb{R}^n)\) and \(\mathcal{P}_{4k, 4\ell}(\mathbb{R}^n)\).



\section{Ancient Solutions to the Biharmonic Heat Equation}
We begin by proving a reverse-Poincaré inequality.
\begin{lemma}\label{reversepoincare}
    Let \(M\) be a complete Riemannian manifold with \(\Ric\) bounded below quadratically with constant \(K\), and consider a function \(u: M \times I \rightarrow \mathbb{R}\) with \(\partial_tu + \Delta\Delta u = 0\). Fix a point \(p \in M\) and let \(B_r = B_r(p)\) and \(Q_r = B_r \times [-r^4, 0]\). For any \(0 < \epsilon < 1\) there is a constant \(c(n, \epsilon, K)\) such that
    \begin{equation}
        \begin{split}
            & r^4 \left(\int_{Q_{\epsilon r}}|\nabla^2u|^2 + r^2\int_{Q_{\epsilon r}}|\nabla \Delta u|^2 \right) + r^8\left(\int_{Q_{\epsilon r}}u_t^2 + r^2\int_{Q_{\epsilon r}}|\nabla u_t|^2\right) \\ & \quad \le c(n, \epsilon, K)\left(\int_{Q_r}u^2 + r^2\int_{Q_r}|\nabla u|^2\right).
        \end{split}
    \end{equation}
\end{lemma}
We proceed by proving the estimate for each term on the left hand side.
\begin{lemma}\label{utboundstep1}
    \begin{equation}
        \begin{split}
            r^4\int_{Q_{\epsilon r}}|\nabla^2u|^2 \le c(n, \epsilon, K)\left(\int_{Q_r}u^2 + r^2\int_{Q_r}|\nabla u|^2\right).
        \end{split}
    \end{equation}
\end{lemma}
\begin{proof}
    Let \(\psi\) be a cutoff function on \(B_R \subset M\) for some \(R > 0\). Using \(u_t = -\Delta \Delta u\), integration by parts, the Bochner formula, and the lower bound on \(\Ric\), we find
    \begin{equation}\label{initial integration by parts}
        \begin{split}
            \partial_t \int_{B_R} u^2\psi^2 & = 2\int_{B_R}uu_t\psi^2 = -2\int_{B_R}(u\psi^2)\Delta \Delta u = -2\int_{B_R}\Delta(u\psi^2)\Delta u
            \\ & = -2\int_{B_R}(\Delta u)^2\psi^2 -2\int_{B_R}u\Delta u \Delta \psi^2 - 4\int \Delta u \langle \nabla u, \nabla \psi^2\rangle 
            \\ & = 2\int_{B_R}\langle \nabla(\psi^2\Delta u), \nabla u \rangle -2\int_{B_R}u\Delta u \Delta \psi^2- 4\int \Delta u \langle \nabla u, \nabla \psi^2\rangle
            \\ & = 2\int_{B_R}\langle \nabla \Delta u, \nabla u\rangle \psi^2 -2\int u\Delta u\Delta \psi^2 - 2\int \Delta u \langle \nabla u, \nabla \psi^2\rangle 
            \\ & = \int \Delta |\nabla u|^2 \psi^2 - 2\int |\nabla^2 u|^2\psi^2 - 2\int \Ric(\nabla u, \nabla u)\psi^2
            \\ & \quad\quad -2\int u\Delta u\Delta \psi^2 - 2\int \Delta u \langle \nabla u, \nabla \psi^2\rangle 
            \\ & \le \int |\nabla u|^2 |\Delta \psi^2| - 2\int |\nabla^2 u|^2\psi^2 + 2\int \frac{K}{R^2}|\nabla u|^2\psi^2
            \\ & \quad\quad -4\int (u\Delta u)(\psi \Delta \psi) -4\int u\Delta u |\nabla \psi|^2 - 4\int \psi\Delta u \langle \nabla u, \nabla \psi\rangle.
        \end{split}
    \end{equation}
    We recall the absorbing inequality: for \(\eta > 0\) and real numbers \(a, b\), we have \(2ab \le \eta a^2 + \eta^{-1}b^2\). Now we apply the Cauchy-Schwarz inequality and the absorbing inequality to the integrals in the last line of (\ref{initial integration by parts}) to find
    \begin{equation}
        \begin{split}
            \partial_t \int_{B_R} u^2\psi^2 & \le \int |\nabla u|^2 |\Delta \psi^2| - 2\int |\nabla^2 u|^2\psi^2 + 2\int \frac{K}{R^2}|\nabla u|^2\psi^2
            \\ & \quad\quad -4\int (u\Delta u)(\psi \Delta \psi) -4\int u\Delta u |\nabla \psi|^2 - 4\int \psi\Delta u \langle \nabla u, \nabla \psi\rangle.
            \\ & \le \int |\nabla u|^2 |\Delta \psi^2| - 2\int |\nabla^2 u|^2\psi^2 + 2\int \frac{K}{R^2}|\nabla u|^2\psi^2
            \\ & \quad\quad +  A\int u^2|\Delta \psi|^2 + \frac{1}{A}\int |\Delta u|^2\psi^2 + B\int u^2|\nabla\psi|^2 + \frac{1}{B}\int |\Delta u|^2|\nabla \psi|^2
            \\ & \quad \quad + C\int|\nabla u|^2|\nabla \psi|^2 + \frac{1}{C}\int |\Delta u|^2\psi^2
        \end{split}
    \end{equation}
    where \(A, B, C > 0\) are quantities which will be determined later. Now we note that \(|\Delta u|^2 \le n|\nabla^2u|^2\) and rearrange terms to get
    \begin{equation}\label{grouped up psi terms in proof of first integral estimate}
        \begin{split}
            \partial_t\int u^2\psi^2 & \le \int |\nabla^2 u|^2\left(-2\psi^2 + \frac{n}{A}\psi^2 + \frac{n}{C}\psi^2 + \frac{n}{B}|\nabla \psi|^2\right) \\ & \quad + \int u^2\left(A|\Delta \psi|^2 + B|\nabla \psi|^2\right) \\ & \quad + \int |\nabla u|^2\left(|\Delta \psi^2| + C|\nabla \psi|^2 + \frac{2K}{R^2}\psi^2\right).
        \end{split}
    \end{equation}
    If we choose \(\psi\) to vanish beyond \(B_{aR}\) for some \(0 < a < 1\), then \(|\nabla \psi|^2\) also vanishes beyond \(B_{aR}\), and we have the estimates
    \[|\psi| \le 1, \quad |\nabla \psi| \le \frac{c(n, a, K)}{R}, \quad |\Delta \psi| \le \frac{c(n, a, K)}{R^2},\]
    which are of course independent from \(u\). We will let \(c(\cdot)\) denote a potentially different constant each time it appears, and also note that
    \[|\Delta \psi^2| = |2\psi \Delta \psi + 2|\nabla \psi|^2| \le 2|\psi||\Delta \psi| + 2|\nabla \psi|^2 \le \frac{c(n, a, K)}{R^2}.\]
    If we choose
    \[A = A'(n, a, K), \quad B = \frac{B'(n, a, K)}{R^2}, \quad C = C'(n, a, K)\]
    for appropriate constants \(A', B', C' > 0\), we can arrange that
    \begin{equation}
        \begin{split}
            -2\psi^2 + \frac{n}{A}\psi^2 + \frac{n}{C}\psi^2 + \frac{n}{B}|\nabla \psi|^2 & \le -\psi^2 + \frac{nR^2}{B'}|\nabla \psi|^2, \\ A|\Delta \psi|^2 + B|\nabla \psi|^2 & \le \frac{A'c(n, a, K)}{R^4} + \frac{B'c(n, a, K)}{R^4}, \\ |\Delta \psi^2| + C|\nabla \psi|^2 & \le \frac{c(n, a, K)}{R^2} + \frac{C'c(n, a, K)}{R^2}.
        \end{split}
    \end{equation}
    Now in equation (\ref{grouped up psi terms in proof of first integral estimate}) we can bound the quantities dependent on \(\psi\) in parentheses to find
    \begin{equation}
        \begin{split}
            & \frac{1}{2}\int_{B_{aR}} |\nabla^2 u|^2 + \partial_t\int_{B_R} u^2\psi^2 
             \le \frac{c(n, a, K)}{R^4}\int_{B_R} u^2 + \frac{c(n, a, K)}{R^2}\int_{B_R} |\nabla u|^2.
        \end{split}
    \end{equation}
    Integrating from \(t = -R^4\) through \(t = 0\), we find that
    \begin{equation}
        \begin{split}
            & \frac{1}{2}\int_{Q_{aR}} |\nabla^2 u|^2 + \int_{t = 0}\int_{B_R} u^2\psi^2 - \int_{t = -R^4}\int_{B_R}u^2\psi^2 \\ &\quad \le \frac{c(n, a, K)}{R^4}\int_{Q_R} u^2 + \frac{c(n, a, K)}{R^2}\int_{Q_R} |\nabla u|^2
        \end{split}
    \end{equation}
    and this gives us
    \begin{equation}\label{definiteintegralhessian}
        \begin{split}
            \int_{Q_{aR}} |\nabla^2 u|^2 \le c(n, a, K)\left(\int_{t = -R^4}\int_{B_R}u^2 + \frac{1}{R^4}\int_{Q_R} u^2 + \frac{1}{R^2}\int_{Q_R} |\nabla u|^2\right).
        \end{split}
    \end{equation}
    Now we use the mean value theorem to bound \(\int_{t = -R^4}\int_{B_R}u^2\), following \cite{colding2021optimal}. Fix \(0 < \epsilon < 1\) and \(r > 0\), as in the lemma statement. For \(0 < a_1 < 1\) there is some \(r_1 \in [a_1r, r]\) such that 
    \begin{equation}
        \begin{split}
            \int_{B_{r_1} \times \{t = -r_1^4\}}u^2 & = \frac{c(a_1)}{r^4}\int_{-r^4}^{-a_1^4r^4}\int_{B_{r_1}}u^2 \le \frac{c(a_1)}{r^4}\int_{-r^4}^{0}\int_{B_{r_1}}u^2 \\ & \le \frac{c(a_1)}{r^4}\int_{Q_r}u^2.
        \end{split}
    \end{equation}
    Choose \(a_1 \in (0, 1)\) such that \(\epsilon < a_1^2\). Replacing \(R\) with \(r_1\) and \(a\) with \(a_1\) in equation (\ref{definiteintegralhessian}) gives us
    \begin{equation}
        \begin{split}
            \int_{Q_{\epsilon r}} |\nabla^2 u|^2 & \le \int_{Q_{a_1r_1}} |\nabla^2 u|^2 \\ & \quad \le c(n, a_1, K)\left(\int_{t = -r_1^4}\int_{B_{r_1}}u^2 + \frac{1}{r_1^4}\int_{Q_{r_1}} u^2 + \frac{1}{r_1^2}\int_{Q_{r_1}} |\nabla u|^2\right) \\ & \quad \le c(n, \epsilon, K)\left(\frac{1}{r^4}\int_{Q_{r}} u^2 + \frac{1}{r^2}\int_{Q_{r}} |\nabla u|^2\right).
        \end{split}
    \end{equation}
    Multiplying through by \(r^4\) proves the lemma.
    
\end{proof}
\begin{lemma}\label{utboundstep2}
    In the same notation as the previous lemma,
    \begin{equation}
        \begin{split}
            r^8\int_{Q_{\epsilon r}}u_t^2 \le c(n, \epsilon, K)\left(\int_{Q_r}u^2 + r^2\int_{Q_r}|\nabla u|^2\right).
        \end{split}
    \end{equation}
\end{lemma}
\begin{proof}
    The proof is essentially the same as that of the previous lemma. For clarity, we make the following initial observation:
    \begin{equation}
        \begin{split}
            \int |\nabla \Delta u|^2\psi^2 & = \int \langle \nabla \Delta u, \nabla \Delta u\rangle\psi^2 = -\int \Delta\Delta u \Delta u \psi^2 - 2\int \psi \Delta u \langle \nabla \Delta u, \nabla \psi\rangle 
            \\ & \le \int u_t\Delta u\psi^2 + 2\int |\Delta u|^2|\nabla \psi|^2 + \frac{1}{2}\int |\nabla \Delta u|^2\psi^2
        \end{split}
    \end{equation}
    so that
    \begin{equation}\label{nabla delta bound}
        \begin{split}
            \int |\nabla \Delta u|^2\psi^2 \le 2\int u_t\Delta u\psi^2 + 4\int |\Delta u|^2|\nabla \psi|^2.
        \end{split}
    \end{equation}
    Now using (\ref{nabla delta bound}), integration by parts, the Cauchy-Schwarz inequality, and the absorbing inequality, we find
    \begin{equation}\label{step2integral1}
        \begin{split}
            \partial_t \int_{B_R} |\Delta u|^2\psi^2 & = -2\int u_t^2 \psi^2 + 8\int u_t\psi\langle \nabla \Delta u, \nabla \psi\rangle + 2\int u_t\Delta u \Delta \psi^2 
            \\ & \le -2\int u_t^2 \psi^2 + A\int u_t^2|\nabla\psi|^2 + \frac{1}{A}\int |\nabla \Delta u|^2\psi^2
            \\ & \quad\quad + 4\int u_t\Delta u\psi\Delta \psi + 4\int u_t\Delta u|\nabla \psi|^2
            \\ & \le -2\int u_t^2 \psi^2 + A\int u_t^2|\nabla\psi|^2 + \frac{4}{A}\int|\Delta u|^2|\nabla \psi|^2
            \\ & \quad\quad + \frac{2B}{A}\int u_t^2\psi^2 + \frac{2}{AB}\int |\Delta u|^2\psi^2
            \\ & \quad\quad + C\int u_t^2\psi^2 + \frac{1}{C}\int |\Delta u|^2|\Delta \psi|^2 
            \\ & \quad\quad + D\int u_t^2|\nabla \psi|^2 + \frac{1}{D}\int |\Delta u|^2|\nabla \psi|^2.
        \end{split}
    \end{equation}
    We choose \(A = A'R^2\), \(B = B'R^2\), \(C = C'\), and \(D = D'R^2\). Then, by rearranging terms, integrating in time, and using the mean value property we find
    \begin{equation}
        \int_{Q_{\epsilon r}}u_t^2 \le \frac{c(n, \epsilon, K)}{r^4}\int_{Q_r}|\Delta u|^2 \le \frac{c(n, \epsilon, K)}{r^4}\int_{Q_r}|\nabla^2 u|^2.
    \end{equation}
    Applying lemma \ref{utboundstep1} completes the proof.
    
\end{proof}
\begin{lemma}\label{utboundstep3}
    In the same notation as the previous lemmas,
    \begin{equation}
        \begin{split}
            r^6\int_{Q_{\epsilon r}}|\nabla \Delta u|^2 \le c(n, \epsilon, K)\left(\int_{Q_r}u^2 + r^2\int_{Q_r}|\nabla u|^2\right).
        \end{split}
    \end{equation}
\end{lemma}
\begin{proof}
    Using the same tools as before,
    \begin{equation}
        \begin{split}
            \partial_t\int |\nabla u|^2\psi^2 & = -2\int u_t\langle \nabla u, \nabla \psi^2\rangle - 2\int u_t\psi^2\Delta u
            \\ & = -4\int u_t\psi\langle \nabla u, \nabla \psi\rangle -2\int \langle \nabla \Delta u, \nabla(\Delta u \psi^2)\rangle
            \\ & \le A\int u_t^2|\nabla \psi|^2 + \frac{1}{A}\int |\nabla u|^2\psi^2 -\int |\nabla \Delta u|^2\psi^2 + 2\int |\Delta u|^2|\nabla \psi|^2.
        \end{split}
    \end{equation}
    Choosing \(A = A'R^4\) and using the mean value theorem then gives
    \begin{equation}
        \begin{split}
            r^6\int_{Q_{\epsilon r}}|\nabla \Delta u|^2 \le c(n, \epsilon, K)\left(r^2\int_{Q_r}|\nabla u|^2 + r^4\int_{Q_r}|\nabla^2 u|^2 + r^8\int_{Q_r}u_t^2\right).
        \end{split}
    \end{equation}
    Applying lemmas \ref{utboundstep1} and \ref{utboundstep2} completes the proof.
    
\end{proof}
\begin{lemma}\label{utboundstep4}
    In the same notation as the previous lemmas, we have
    \begin{equation}
        \begin{split}
            r^{10}\int_{Q_{\epsilon r}}|\nabla u_t|^2 \le c(n, \epsilon, K)\left(\int_{Q_{r}}u^2 + r^2\int_{Q_{r}}|\nabla u|^2\right).
        \end{split}
    \end{equation}
\end{lemma}
\begin{proof}
    We start by noting that because \(\partial_t\) and \(\Delta\) commute, \(\Delta u\) is a solution to the biharmonic heat equation if \(u\) is. Thus by the previous lemmas we have
    \begin{equation}
        \begin{split}
            r^{10}\int_{Q_{\epsilon^2 r}}|\nabla u_t|^2 & = r^{10}\int_{Q_{\epsilon^2 r}}|\nabla \Delta \Delta u|^2 \\ &\le c(n, \epsilon, K)\left(r^4\int_{Q_{\epsilon r}}|\Delta u|^2 + r^6\int_{Q_{\epsilon r}}|\nabla \Delta u|^2\right)\\ & \le c(n, \epsilon, K)\left(\int_{Q_{r}}u^2 + r^2\int_{Q_{r}}|\nabla u|^2\right).
        \end{split}
    \end{equation}
    
\end{proof}
\begin{proof}[Proof of lemma \ref{reversepoincare}]
    The inequality in lemma \ref{reversepoincare} is the sum of the inequalities in lemmas \ref{utboundstep1}, \ref{utboundstep2}, \ref{utboundstep3}, and \ref{utboundstep4} (with the expression \(c(n, \epsilon, K)\) potentially representing different constants each time it appears, as usual).
    
\end{proof}

\section{Bounding the dimension of \(\mathcal{P}_{4k, 4\ell}(M)\)}
\begin{lemma}\label{exponentbound}
     Suppose that \(M\) has polynomial volume growth, i.e. that \(\operatorname{Vol}(B_r) \le C(1 + R)^{d_V}\) with fixed constants \(C, d_V > 0\) for all \(r > 0\). If \(u \in \mathcal{P}_{d, d'}(M)\), then \(\partial_t^k u\) is identically 0 if \(8k > 2d + 2d' + d_V + 6\).
\end{lemma}
\begin{proof}
    Because \(\partial_t\) and \(\Delta\) commute, \(\partial_t^ju\) solves the biharmonic heat equation for every \(j\). It follows from iterating lemma \ref{reversepoincare} that
    \begin{equation}
        \begin{split}
            \int_{Q_{r/10^k}} |\partial_t^ku|^2 + r^2\int_{Q_{r/10^k}}|\nabla \partial_t^k u|^2 \le \frac{c(n, k, K)}{r^{8k}}\left(\int_{Q_r}u^2 + r^2\int_{Q_r}|\nabla u|^2\right).
        \end{split}
    \end{equation}
    In particular
    \begin{equation}
        \begin{split}
            \int_{Q_{r/10^k}} |\partial_t^ku|^2 &\le \frac{c(n, k, K)}{r^{8k}}\int_{Q_r}u^2 + \frac{c(n,k, K)}{r^{8k - 2}}\int_{Q_r}|\nabla u|^2 \\ & \le c(n,k, K)\left(\operatorname{Vol}(B_r)r^4\right)\left(r^{- 8k}\sup_{Q_r}u^2 + r^{2 - 8k}\sup_{Q_r}|\nabla u|^2\right) \\ & \le c(n, k, K, u, M)r^{4 - 8k}(1 + r)^{d_V}\left((1 + r)^{2d} + r^{2}(1 + r)^{2d'}\right).
        \end{split}
    \end{equation}
    When \(8k > 2d + 2d' + d_V + 6\), taking the limit as \(r \rightarrow \infty\) shows that \(\partial_t^ku\) must be identically 0.
    
\end{proof}
\begin{lemma}\label{coefficientbound}
    Suppose \(u \in \mathcal{P}_{4k, 4\ell}(M)\). Let \(d = \min\{k, \ell + 1\}\). Then \(u\) can be written as \(u = p_0(x) + tp_1(x) + \cdots + t^dp_d(x)\), with
    \begin{equation}\label{pde for pj statement in lemma}
        \Delta\Delta p_{d} = 0 \quad \text{and} \quad \Delta\Delta p_j = -(j+1)p_{j + 1} \; (j < d).
    \end{equation}
    Furthermore,
    \begin{equation}\label{coefficients in polynomial expression bounds}
        |p_j(x)| \le C_j(1 + |x|)^{4(k - j)}, \quad |\nabla p_j(x)| \le C_j'(1 + |x|)^{4(\ell - j)} \quad (j < \ell + 1), \quad \nabla p_{\ell + 1}(x) = 0.
    \end{equation}
\end{lemma}
\begin{proof}
    As in lemma \ref{exponentbound}, choose \(8m > 8k + 8\ell + d_V + 6\). Then \(\partial_t^mu = 0\) for any \(u \in \mathcal{P}_{k, \ell}(M)\). It follows that for any \(d > m\), we can write
    \begin{equation}
        \begin{split}
            u(x, t) = p_0(x) + tp_1(x) + \ldots + t^{d}p_{d}(x).
        \end{split}
    \end{equation}
    We now refine this to \(d \ge k\). Fix an arbitrary \(x \in M\). Then for each \(j\) as \(t \rightarrow -\infty\) we have
    \begin{equation}\label{initial bound on p and nabla p}
        |u(x, t)| \ge O(|p_j(x)t^j|) \quad \text{and} \quad |\nabla u(x, t)| \ge O(|\nabla p_{j}(x)t^{j}|).
    \end{equation}
    The polynomial growth bounds on \(u\) and \(\nabla u\) show that we must have \(p_j(x) = 0\) if \(j > k\) and \(\nabla p_j(x) = 0\) if \(j > \ell\).

    It follows that we can take \(d \ge k\) in our expression \(u\) (with some of the coefficient functions possibly being zero).

    We now show equation (\ref{pde for pj statement in lemma}) and use it to show \(d \ge \ell + 1\). Because \(u\) satisfies the biharmonic heat equation, we have
    \begin{equation}\label{pde for pj derivation}
        \begin{split}
            \partial_tu + \Delta\Delta u & = jt^{j - 1}p_j(x) + t^{j}\Delta\Delta p_j(x) \\ & = t^{j - 1}(jp_j + \Delta\Delta p_{j - 1}) \\ & = 0,
        \end{split}
    \end{equation}
    so that \(\Delta\Delta p_{d} = 0\) and for \(j < d\)
    \begin{equation}\label{pde for pj}
        \Delta\Delta p_j = -(j+1)p_{j + 1}.
    \end{equation}
    From equation (\ref{initial bound on p and nabla p}) we deduced that \(\nabla p_j = 0\) if \(j > \ell\). Thus, \(p_j\) is constant for \(j > \ell\). It now follows from equation (\ref{pde for pj}) that if \(j \ge \ell + 2\), then \(p_j = 0\). Thus, we can take \(d \ge \ell + 1\). In particular, if \(j \ge \min\{k, \ell + 1\}\), then \(p_j(x) = 0\), and so we fix \(d = \min\{k, \ell + 1\}\).

    Now we show equation (\ref{coefficients in polynomial expression bounds}). We first show that the \(p_j\) grow polynomially of degree at most \(4k\) and the \(\nabla p_j\) grow polynomially of degree at most \(4\ell\). Fix distinct numbers \(-1 < t_1 < \cdots < t_{d + 1} < -1/2\). We claim that the vectors
    \begin{equation}
        (1, t_i, t_i^2, \ldots, t_i^{d})
    \end{equation}
    are linearly independent in \(\mathbb{R}^{d + 1}\). If they were not, then they would lie in a strict subspace of \(\mathbb{R}^{d + 1}\) and there would thus be a vector \((a_0, \ldots, a_{d})\) orthogonal to all of them. That is,
    \begin{equation}
        a_0 + a_1t_i + a_2t_i^2 + \cdots + a_{d + 1}t_i^{d} = 0
    \end{equation}
    for \(i = 1, \ldots, d+ 1\). But a polynomial of degree \(d\) can have at most \(d\) distinct roots, so this is a contradiction.

    Since there are \(d + 1\) vectors \((1, t_i, \ldots, t_i^{d + 1})\), they span \(\mathbb{R}^{d + 1}\), and so there are constants \(b_i^j\) such that
    \begin{equation}\label{bdefinitions}
        e_j = b_i^j(1, t_i, \ldots, t_i^{d}).
    \end{equation}
    It now follows that
    \begin{equation}\label{expression of p in b}
        p_j(x) = b_i^ju(x, t_i), \quad \nabla p_j(x) = b_i^j\nabla u(x, t_i),
    \end{equation}
    and we conclude that \(p_j\) can grow at most polynomially of degree \(4k\) and \(\nabla p_j\) can grow at most polynomially of degree \(4\ell\). Because \(p_j\) vanishes when \(j > k\) and \(\nabla p_j\) vanishes when \(j > \ell + 1\), we have
    \begin{equation}
        u = p_0 + t p_1 + \ldots + t^{k}p_{k} \quad \text{and} \quad \nabla u = \nabla p_0 + t\nabla p_1 + \ldots + t^{\ell}\nabla p_{\ell},
    \end{equation}
    it follows that
    \begin{equation}
        \begin{split}
            |u(x, t)| \le C(1 + |t|^k + |x|^{4k}) \quad \text{and} \quad |\nabla u(x, t)| \le C(1 + |t|^{\ell} + |x|^{4\ell})
        \end{split}
    \end{equation}
    From equation (\ref{bdefinitions}) we have
    \begin{equation}
        \begin{split}
            \sum_i b_i^ju(x, R^4t_i) & = \sum_i\sum_{m}b_i^jp_{m}(x)R^{4j}t_i^{m} = R^{4j}\sum_mp_m(x)\left(\sum_ib_{i}^jt_i^m\right) \\ & = \sum R^{4j}\sum_m p_m(x)\delta_{mj} \\ & = R^{4j}p_j(x).
        \end{split}
    \end{equation}
    Similarly,
    \begin{equation}
        \sum_i b_i^j\nabla u(x, R^4t_i) = R^{4j}\nabla p_j(x)
    \end{equation}
    Thus
    \begin{equation}
        \begin{split}
            |R^{4j}p_j(x)| & = \left|\sum_i b_i^ju(x, R^4t_i)\right| \le A\sum_i\left|u(x, R^4t_i)\right| \\ & \le A(1 + |x|^{4k} + \sum_i|Rt_i|^{4d}) \le AR^{4k}
        \end{split}
    \end{equation}
    and similarly
    \begin{equation}
        \begin{split}
            |R^{4j}\nabla p_j(x)| & \le A'(1 + |x|^{4\ell} + \sum_i|Rt_i|^{4\ell}) \le A'R^{4\ell},
        \end{split}
    \end{equation}
    so that \(|p_j(x)| \le A_jR^{4(k - j)}\) and \(|p_j(x)| \le A_j'R^{4(\ell - j)}\).
\end{proof}
\begin{proof}[Proof of Theorem \ref{dimension bound theorem}]
    Choose some \(u \in {P}_{4k, 4\ell}(M)\) and suppose \(u = p_0(x) + tp_1(x) + \cdots + t^dp_d(x)\), as in lemma \ref{coefficientbound}. Then \(\Delta\Delta p_{d} = 0\) and for \(j < d\), \(\Delta\Delta p_j = -(j+1)p_{j + 1}\). Thus, there a linear map \(\Psi_0: \mathcal{P}_{4k, 4\ell} \rightarrow \mathcal{H}_{4k, 4\ell}\) defined by \(\Psi_0u = p_{d}\) (here we use the coefficient estimate in (\ref{coefficients in polynomial expression bounds})). If we let \(\mathcal{K}_0 = \ker \Psi_0\) we find
    \begin{equation}
        \dim \mathcal{P}_{4k, 4\ell} \le \dim \mathcal{K}_0 + \dim \mathcal{H}_{4k, 4\ell}
    \end{equation}
    If \(u \in \mathcal{K}_0\), then \(p_{d} = 0\) and \(\Delta\Delta p_{d - 1} = -dp_{d} = 0\), and so we have a map \(\Psi_1: \mathcal{K}_0 \rightarrow \mathcal{H}_{4(k - 1), 4(\ell - 1)}\) defined by setting \(\Psi_1u = p_{d - 1}\). Letting \(\ker \Psi_1 = \mathcal{K}_1\) then
    \begin{equation}
        \dim \mathcal{K}_0 \le \dim \mathcal{K}_1 + \dim \mathcal{H}_{4k, 4\ell}
    \end{equation}
    When \(k \le \ell + 1\), we can repeat this \(k\) times to get
    \begin{equation}
        \dim \mathcal{P}_{4k, 4\ell}(M) \le \sum_{i = 0}^{k} \dim \mathcal{H}_{4(k - i), 4(\ell - i)}(M)
    \end{equation}
    When \(k > \ell + 1\), we have from lemma \ref{coefficientbound} that \(\nabla p_{\ell + 1} = 0\), and so \(p_{\ell + 1} = p_d\) is a constant. Thus in this case \(p_d\) lies in a one dimensional subspace of \(H_{4k, 4\ell}(M)\), so that
    \begin{equation}
        \dim \mathcal{P}_{4k, 4\ell} = \dim \mathcal{K}_0 + 1.
    \end{equation}
    We then iterate the same argument as before, \(\ell\) times, to get the second inequality in (\ref{equation in the first statement of the main theorem}).
    
\end{proof}

\subsection{Biharmonic Polynomials in \(\mathbb{R}^n\)}
Our goal now is to show that the inequalities in Theorem \ref{dimension bound theorem} are sharp in \(\mathbb{R}^n\). We start by showing that the solutions \(u : \mathbb{R}^n \rightarrow \mathbb{R}\) to the biharmonic equation \(\Delta\Delta u = 0\) are polynomials. We start by showing what is essentially a reverse-Poincaré inequality for biharmonic functions:

\begin{lemma}\label{reverse poincare for just biharmonic equation}
    Let \(M\) be a manifold with \(\Ric\) bounded below quadratically with constant \(K\) and suppose \(\Delta\Delta u = 0\). Then for \(\epsilon \in (0, 1)\) and \(r > 0\).
    \begin{equation}
        \begin{split}
            r^4\left(\int_{B_{\epsilon r}(p)}|\nabla^2u|^2 + r^2\int_{B_{\epsilon r}(p)}|\nabla \Delta u|^2\right) \le c(n, \epsilon, K)\left(\int_{B_r(p)}u^2 + r^2\int_{B_r(p)}|\nabla u|^2\right)
        \end{split}
    \end{equation}
\end{lemma}
\begin{proof}
    One proof is to note that the function \(u(x, t) = u(x)\) is indeed an ancient solution to \(\partial_tu + \Delta^2u = 0\), and then the result follows from lemma \ref{reversepoincare}. One can also use integration by parts, the Bochner formula, and the absorbing inequality as in our previous lemmas (similarly to how one would prove reverse-Poincaré for harmonic functions).
    
\end{proof}

\begin{proposition}\label{polynomially bounded in Rn makes you a polynomial}
    Let \(u \in \mathcal{H}_{d, d'}(\mathbb{R}^n)\). Then \(u\) is a polynomial.
\end{proposition}
\begin{proof}
    Since the coordinate derivatives \(\partial_{x_i}\) commute with \(\Delta\) in \(\mathbb{R}^n\), we have for all \(k > 0\)
    \begin{equation}
        \begin{split}
            \int_{B_{\epsilon r}(p)}|\nabla^2 \partial_{x_i}^ku|^2 + r^2\int_{B_{\epsilon r}(p)}|\nabla \Delta \partial_{x_i}^ku|^2 & \le \frac{c(n)}{r^{4k}}\left(\int_{B_r(p)}u^2 + r^2\int_{B_r(p)}|\nabla u|^2\right) \\ & \le \frac{c(n)}{r^{4k}}\left(Cr^n(1 + r)^{2d} + C'r^{n + 2}(1 + r)^{2d'}\right)
        \end{split}
    \end{equation}
    For \(4k > 2n + 2d + 2d' + 2\), the quantity on the right hand side goes to \(0\) as \(r \rightarrow \infty\). Thus, there is some \(K\) so that for \(k \ge K\) we have \(\partial_{x_i}^k u = 0\) everywhere. We can carry out this argument for any \(x_i\), and so we conclude that \(u\) is a polynomial.
\end{proof}
\begin{corollary}
    If \(u \in \mathcal{H}_{k, \ell}(\mathbb{R}^n)\), then there is some \(d\) such that \(u \in \mathcal{H}_{d, d-1}(\mathbb{R}^n)\).
\end{corollary}\label{fewer spaces in R^n}
\begin{proof}
    This follows instantly from the fact that a polynomial's derivative in \(\mathbb{R}^n\) grows polynomially of one degree lower.
    
\end{proof}

We will now discuss in more detail the biharmonic polynomials in \(\mathbb{R}^n\). Due to corollary \ref{fewer spaces in R^n} we can consider just the spaces \(\mathcal{H}_{d, d-1}(\mathbb{R}^n)\). Following \cite{colding2021optimal}, let \(A_j^n\) be the set of homogeneous polynomials in \(\mathbb{R}^n\) of degree \(j\). Then \(\Delta: A_j^n \rightarrow A_{j - 2}^n\) is a linear map. From \cite{colding2021optimal} we have
\begin{lemma}\label{colding minicozzi lemma}
    For each \(d\), the map \(\Delta: A_{d + 2}^n \rightarrow A_d^n\) is onto.
\end{lemma}
\begin{lemma}\label{dim for biharmonic polynomials}
    Consider the map \(\Delta\Delta: A_{d + 4}^n \rightarrow A_d^n\) for \(d > 0\). Let \(B_d^n\) be the kernel of \(\Delta\Delta: A_d^n \rightarrow A_{d - 4}^n\). For each \(d > 0\), \(\Delta\Delta: A_{d + 4}^n \rightarrow A_d^n\) is onto, \(\dim B_d^n = \dim A_d^n - A_{d - 4}^n\), and
    \begin{equation}
        \dim \mathcal{H}_{d, d-1}(\mathbb{R}^n) = \dim A_d^n + \dim A_{d - 1}^n + \dim A_{d - 2}^n + \dim A_{d - 3}^n
    \end{equation}
\end{lemma}
\begin{proof}
    To show the map is onto we use the previous lemma. Fix \(p \in A_d^n\). Then there is \(p' \in A_{d + 2}^n\) such that \(\Delta p' = p\). Furthermore, there is \(p'' \in A_{d + 4}^n\) such that \(\Delta p'' = p'\). Thus \(\Delta\Delta p'' = p\).
    
    The fact that \(\dim B_d^n = \dim A_d^n - A_{d - 4}^n\) follows from the fact that \(\Delta\Delta: A_d^n \rightarrow A_{d - 4}^n\) is onto. For the last claim, we note that \(\mathcal{H}_{d, d-1}(\mathbb{R}^n)\) is the direct sum of the spaces \(\mathcal{H}_{j, j-1}(\mathbb{R}^n) \cap A_j^n = B_j^n\) for \(j \le d\). Summing \(\dim B_j^n = \dim A_j^n - A_{j - 4}^n\) over \(j\) gives the second claim.
    
\end{proof}
\subsection{Bicaloric Polynomials in \(\mathbb{R}^n\)}
Now we consider polynomially bounded solutions to the biharmonic heat equation in \(\mathbb{R}^n\).
\begin{proposition}\label{bicaloric functions in Rn are polynomials}
    Let \(u \in \mathcal{P}_{k, \ell}(\mathbb{R}^n)\). Then \(u\) is a polynomial in \(x_i\) and \(t\).
\end{proposition}
\begin{proof}
    As before, this follows from the reverse Poincaré estimate and the fact that the operators \(\partial_{x_i}\), \(\partial_t\), and \(\Delta\) commute in \(\mathbb{R}^n\).
\end{proof}
\begin{corollary}
    If \(u \in \mathcal{P}_{k, \ell}(\mathbb{R}^n)\), then there is some \(d\) such that \(u \in \mathcal{P}_{d, d-1}(\mathbb{R}^n)\).
\end{corollary}\label{fewer caloric spaces in R^n}
\begin{proof}
    This follows the same way as before.
\end{proof}
Given a monomial in \(x_i\) and \(t\), we define its biparabolic degree as follows: for \(t^{n_0}\prod x_i^{n_i}\), the biparabolic degree is \(4n_0 + \sum n_i\). The degree of a polynomial is then the maximal degree of the monomials summing to it. Let \(\mathcal{A}_j^n\) be the set of homogeneous polynomials in \(\mathbb{R}^n\) of biparabolic degree \(j\). We have
\begin{equation}
    \mathcal{A}_d^n = A_d^n \oplus t A_{d - 4}^n \oplus t^2A_{d - 8}^n \oplus \cdots
\end{equation}
\begin{lemma}\label{dim for bicaloric polynomials}
    For \(d > 0\) we have \(\dim(\mathcal{P}_{d, d-1}(\mathbb{R}^n) \cap \mathcal{A}_d^n) = \dim A_d^n\) and
    \begin{equation}
        \dim \mathcal{P}_{d, d-1}(\mathbb{R}^n) = \sum_{j = 0}^d\dim A_j^n.
    \end{equation}
\end{lemma}
\begin{proof}
    Both \(\partial_t\) and \(\Delta\Delta\) map \(\mathcal{A}_d^n\) to \(\mathcal{A}_{d-4}^n\). We note that for \(u \in \mathcal{A}_{d-4}^n\) we have
    \begin{equation}
        (\partial_t + \Delta\Delta)\left(tu - \frac{1}{2}t^2(\partial_t + \Delta\Delta)u + \frac{1}{6}t^3(\partial_t + \Delta\Delta)^2 u - \cdots\right) = u,
    \end{equation}
    so that \(\partial_t + \Delta\Delta\) is surjective. Thus,
    \begin{equation}
        \dim (\mathcal{P}_{d, d-1}(\mathbb{R}^n) \cap \mathcal{A}_d^n) = \dim \mathcal{A}_d^n - \dim\mathcal{A}_{d-4}^n = \dim A_d^n
    \end{equation}
    since \(\mathcal{P}_{d, d-1}(\mathbb{R}^n) \cap \mathcal{A}_d^n\) is the kernel of \(\partial_t + \Delta\Delta\) restricted to \(\mathcal{A}_d^n\). Summing gives the second claim.
    
\end{proof}
Now finally we show that the estimate in Theorem \ref{dimension bound theorem} is sharp in \(\mathbb{R}^n\).
\begin{corollary}\label{sharp in Rn}
    For positive integers \(d > 0\)
    \begin{equation}
        \dim \mathcal{P}_{4d, 4d - 1}(\mathbb{R}^n) = \sum_{i = 0}^{d} \dim \mathcal{H}_{4(d - i), 4(d - i) - 1}(\mathbb{R}^n).
    \end{equation}
\end{corollary}
\begin{proof}
    From lemmas \ref{dim for bicaloric polynomials} and \ref{dim for biharmonic polynomials} we have
    \begin{equation}
        \begin{split}
            \dim \mathcal{P}_{4d, 4d-1}(\mathbb{R}^n) & = \sum_{j = 0}^{4d}\dim A_j^n = \sum_{j = 0}^d(\dim A_{4j}^n + \dim A_{4j - 1}^n + \dim A_{4j - 2}^n + \dim A_{4j - 3}^n) \\ & = \sum_{j = 0}^d\dim \mathcal{H}_{4d, 4d - 1}(\mathbb{R}^n).
        \end{split}
    \end{equation}
    Setting \(k = \ell = d\) in Theorem \ref{dimension bound theorem} and noting \(\mathcal{P}_{4d, 4d}(\mathbb{R}^n) = \mathcal{P}_{4d, 4d-1}(\mathbb{R}^n)\) and \(\mathcal{H}_{4d, 4d}(\mathbb{R}^n) = \mathcal{H}_{4d, 4d-1}(\mathbb{R}^n)\) shows that the estimate in Theorem \ref{dimension bound theorem} is sharp.
    
\end{proof}
\subsubsection*{Acknowledgements}
I would like to thank my advisor William Minicozzi for introducing this problem to me and for his support along the way.

\newpage
\printbibliography

@article{colding2021optimal,
  title={Optimal bounds for ancient caloric functions},
  author={Colding, Tobias Holck and Minicozzi II, William P},
  journal={Duke Mathematical Journal},
  volume={170},
  number={18},
  pages={4171--4182},
  year={2021},
  publisher={Duke University Press}
}

@article{colding1997harmonic,
  title={Harmonic functions on manifolds},
  author={Colding, Tobias H and Minicozzi, William P},
  journal={Annals of mathematics},
  volume={146},
  number={3},
  pages={725--747},
  year={1997},
  publisher={JSTOR}
}

@article {yau1975harmonic,
    AUTHOR = {Yau, Shing Tung},
     TITLE = {Harmonic functions on complete {R}iemannian manifolds},
   JOURNAL = {Comm. Pure Appl. Math.},
  FJOURNAL = {Communications on Pure and Applied Mathematics},
    VOLUME = {28},
      YEAR = {1975},
     PAGES = {201--228},
      ISSN = {0010-3640,1097-0312},
   MRCLASS = {53C20 (31C05)},
  MRNUMBER = {431040},
MRREVIEWER = {Yoshiaki\ Maeda},
       DOI = {10.1002/cpa.3160280203},
       URL = {https://doi.org/10.1002/cpa.3160280203},
}

@article {Lin1998classification,
    AUTHOR = {Lin, Chang-Shou},
     TITLE = {A classification of solutions of a conformally invariant
              fourth order equation in {${\bf R}^n$}},
   JOURNAL = {Comment. Math. Helv.},
  FJOURNAL = {Commentarii Mathematici Helvetici},
    VOLUME = {73},
      YEAR = {1998},
    NUMBER = {2},
     PAGES = {206--231},
      ISSN = {0010-2571,1420-8946},
   MRCLASS = {35J60 (53C21)},
  MRNUMBER = {1611691},
MRREVIEWER = {John\ Urbas},
       DOI = {10.1007/s000140050052},
       URL = {https://doi.org/10.1007/s000140050052},
}

@article {chang1995extremal,
    AUTHOR = {Chang, Sun-Yung A. and Yang, Paul C.},
     TITLE = {Extremal metrics of zeta function determinants on
              {$4$}-manifolds},
   JOURNAL = {Ann. of Math. (2)},
  FJOURNAL = {Annals of Mathematics. Second Series},
    VOLUME = {142},
      YEAR = {1995},
    NUMBER = {1},
     PAGES = {171--212},
      ISSN = {0003-486X,1939-8980},
   MRCLASS = {58E11 (58G26)},
  MRNUMBER = {1338677},
MRREVIEWER = {Rafe\ Mazzeo},
       DOI = {10.2307/2118613},
       URL = {https://doi.org/10.2307/2118613},
}

@article {kuwert2002gradient,
    AUTHOR = {Kuwert, Ernst and Sch\"{a}tzle, Reiner},
     TITLE = {Gradient flow for the {W}illmore functional},
   JOURNAL = {Comm. Anal. Geom.},
  FJOURNAL = {Communications in Analysis and Geometry},
    VOLUME = {10},
      YEAR = {2002},
    NUMBER = {2},
     PAGES = {307--339},
      ISSN = {1019-8385,1944-9992},
   MRCLASS = {53C44 (35K10)},
  MRNUMBER = {1900754},
MRREVIEWER = {Shu-Cheng\ Chang},
       DOI = {10.4310/CAG.2002.v10.n2.a4},
       URL = {https://doi.org/10.4310/CAG.2002.v10.n2.a4},
}

@incollection {willmore2000surfaces,
    AUTHOR = {Willmore, T. J.},
     TITLE = {Surfaces in conformal geometry},
      NOTE = {Special issue in memory of Alfred Gray (1939--1998)},
   JOURNAL = {Ann. Global Anal. Geom.},
  FJOURNAL = {Annals of Global Analysis and Geometry},
    VOLUME = {18},
      YEAR = {2000},
    NUMBER = {3-4},
     PAGES = {255--264},
      ISSN = {0232-704X,1572-9060},
   MRCLASS = {53C42 (37K25)},
  MRNUMBER = {1795097},
MRREVIEWER = {William\ P.\ Minicozzi, II},
       DOI = {10.1023/A:1006717506186},
       URL = {https://doi.org/10.1023/A:1006717506186},
}

@article {lamm2005biharmonic,
    AUTHOR = {Lamm, Tobias},
     TITLE = {Biharmonic map heat flow into manifolds of nonpositive
              curvature},
   JOURNAL = {Calc. Var. Partial Differential Equations},
  FJOURNAL = {Calculus of Variations and Partial Differential Equations},
    VOLUME = {22},
      YEAR = {2005},
    NUMBER = {4},
     PAGES = {421--445},
      ISSN = {0944-2669,1432-0835},
   MRCLASS = {53C44 (58E20)},
  MRNUMBER = {2124627},
MRREVIEWER = {Roger\ Moser},
       DOI = {10.1007/s00526-004-0283-8},
       URL = {https://doi.org/10.1007/s00526-004-0283-8},
}

@article {kuwert2004removable,
    AUTHOR = {Kuwert, Ernst and Sch\"{a}tzle, Reiner},
     TITLE = {Removability of point singularities of {W}illmore surfaces},
   JOURNAL = {Ann. of Math. (2)},
  FJOURNAL = {Annals of Mathematics. Second Series},
    VOLUME = {160},
      YEAR = {2004},
    NUMBER = {1},
     PAGES = {315--357},
      ISSN = {0003-486X,1939-8980},
   MRCLASS = {58E12 (53C44)},
  MRNUMBER = {2119722},
MRREVIEWER = {Shu-Cheng\ Chang},
       DOI = {10.4007/annals.2004.160.315},
       URL = {https://doi.org/10.4007/annals.2004.160.315},
}

@article{colding2020complexity,
  title={Complexity of parabolic systems},
  author={Colding, Tobias Holck and Minicozzi, William P},
  journal={Publications math{\'e}matiques de l'IH{\'E}S},
  volume={132},
  number={1},
  pages={83--135},
  year={2020},
  publisher={Springer}
}

@article{colding2019liouville,
  title={Liouville properties},
  author={Colding, Tobias Holck and Minicozzi II, William P},
  journal={arXiv preprint arXiv:1902.09366},
  year={2019}
}

@article{colding2019search,
  title={In search of stable geometric structures},
  author={Colding, Tobias Holck and Minicozzi II, William P},
  journal={arXiv preprint arXiv:1907.03672},
  year={2019}
}

@misc{wang2025qualitativebehaviorbiharmonicfunctions,
      title={The qualitative behavior for biharmonic functions on open manifolds}, 
      author={Lin Wang and Miaomiao Zhu},
      year={2025},
      eprint={2511.09393},
      archivePrefix={arXiv},
      primaryClass={math.DG},
      url={https://arxiv.org/abs/2511.09393}, 
}

\end{document}